\documentclass[11pt]{article}
\usepackage{amsmath,amsfonts,amssymb,amsthm}
\usepackage{graphicx}
\usepackage{bbm}
\usepackage[body={13cm,19cm}]{geometry}
\usepackage{pstricks}
\usepackage[authoryear,longnamesfirst]{natbib}
\usepackage[center]{crop}
\usepackage{mathrsfs}

\def\note#1{\par\smallskip%
\noindent%
\llap{$\boldsymbol\Longrightarrow$}%
\fbox{\vtop{\parindent=0cm\small #1}}%
\rlap{$\boldsymbol\Longleftarrow$}%
\par\smallskip}
\def\note#1{}
\numberwithin{equation}{section}

\newtheorem{theorem}{Theorem}[section]

\newtheorem{corollary}[theorem]{Corollary}
\newtheorem{lemma}[theorem]{Lemma}

\theoremstyle{definition}
\newtheorem{remark}[theorem]{Remark}

\newcommand{\norm}[1]{\|#1\|}                    

\newcommand{\ab}[1]{\vert#1\vert}                

\newcommand{\exponent}[1]{\exp\{#1\}}            
\newcommand{\Exponent}[1]{\exp\Bigl\{#1\Bigr\}}  

\def\eq#1{\eqref{#1}}

\newcommand{\ee}{e}
\newcommand{\dtv}{d_{\mathrm{TV}}}
\newcommand{\dloc}{d_{\mathrm{loc}}}
\def\ahalf{{\textstyle\frac12}}
\def\D{\Delta}
\def\T{\Theta}
\DeclareMathOperator{\Bi}{Bi}
\DeclareMathOperator{\N}{N}

\DeclareMathOperator{\Var}{Var}
\DeclareMathOperator{\IE}{\mathbbm{E}}
\def\IE{\mathbbm{E}}
\DeclareMathOperator{\IZ}{\mathbbm{Z}}
\DeclareMathOperator{\IR}{\mathbbm{R}}
\DeclareMathOperator{\IP}{\mathbbm{P}}
\def\law{{\mathscr{L}}}

\makeatletter

\renewcommand\section{\@startsection {section}{1}{\z@}%
{-3.5ex \@plus -1ex \@minus -.2ex}%
{1.3ex \@plus.2ex}%
{\center\small\sc\mathversion{bold}\MakeUppercase}}

\makeatother

\begin{document}

\title{\sc\bf\large\MakeUppercase{A Three-Parameter Binomial Approximation}}
\author{%
\sc
Vydas \v Cekanavi\v cius\footnote{Vilnius University,
Faculty of Mathematics and Informatics,
Naugarduko 24, Vilnius LT-03223},
Erol A. Pek\"{o}z\footnote{(Corresponding Author)
Boston University School of Management,
595 Commonwealth Avenue, Boston, MA 02215,
Phone: 617-353-2676, Email: pekoz@bu.edu},\\[0.5ex]
\sc
Adrian R\"ollin\footnote{
National University of Singapore,
2 Science Drive 2, 117543 Singapore}\,
and
Michael Shwartz\footnote{
Center for Organization, Leadership and Management Research,
Veterans' Health Administration, Boston, MA, and Boston University School of Management,
595 Commonwealth Avenue, Boston, MA 02215}
}
\date{\small Version from \today}
\maketitle

\begin{abstract}
We approximate the distribution of the sum of independent but not necessarily identically distributed
Bernoulli random variables using a shifted binomial distribution where the three
parameters (the number of trials, the probability of success, and the shift amount) are
chosen to match up the first three moments of the two
distributions.  We give a bound on the approximation error in terms of the total
variation metric using Stein's method. A numerical study is discussed that shows
shifted binomial approximations typically are more accurate than Poisson or
standard binomial approximations. The application of the approximation to
solving a problem arising in Bayesian hierarchical modeling is also discussed.

\end{abstract}

\section{Introduction}

A common method for improving the accuracy of an approximation is the
construction of an asymptotic expansion.  In practice, however, this can be more
time consuming and much less convenient than calculating the values of a known
distribution. An alternative approach is thus to modify a common approximating
distribution by introducing some new parameters which then can be used to
achieve a better fit.  The use of common distributions can make it easy to avoid the need for specialized programming when using standard statistical packages to model data.

One of the  simplest modifications is shifting, and this approach works well in
the Poisson case. For a large number of independent rare events, the
distribution of the number of them that occur is often well approximated by a
Poisson distribution.  If some of the events are not in fact so rare, this
approximation is likely to be poor: the expected number of events occurring may
not be close to the variance, but these are equal for the Poisson distribution.
One easy way to address this problem is to introduce a shift by adding or
subtracting a constant from the Poisson random variable.  This then gives
essentially two parameters that can be fitted to match the first two moments
(subject to the constraint that the shift is an integer).  Shifted (also
referred to as translated or centered) Poisson approximation has been studied in
many papers: see, for example, \cite{Cekanavicius2001}, \cite{Barbour2002},
\cite{Rollin2005}, \cite{Barbour2006}, and references therein.

One of the goals of this paper is to investigate the effect of shifting applied
to a two-parameter distribution. It is clear that shifting changes a
distribution's mean but not its variance and higher centered moments. Can we
expect by shifting to conveniently obtain a three-parameter distribution and
match three corresponding moments? In the case of normal approximation, the
obvious answer is no. The normal distribution already has a parameter which can
be treated as shifting. Since both parameters of two-parameter distributions are
usually closely related to their first two moments, it seems important to show
that there are natural cases where shifting can be successfully applied. Below
we use shifted (centered, translated)  binomial approximation for the sum of
Bernoulli variables. Our primary interest for the statistical application we
consider is in the case when the variables are independent.

 In the literature, the distribution of the sum of independent Bernoulli random
variables with not-necessarily-identical probabilities is called a Poisson-binomial
distribution.
%
This distribution is widely applicable and widely studied, and bounds on
approximation errors for various approximations have been developed.  See
\cite{Chen1997} for an overview of the Poisson-binomial distribution and
\cite{Pitman1997} for applications, as well as \cite{LeCam1960} and
\cite{Barbour1992} for some Poisson approximation results. A number of
researchers have studied the binomial distribution as an approximation for the
Poisson-binomial distribution. For example, \cite{Choi2002} argue that binomial
approximations are better than Poisson approximations.

Before discussing some previously obtained results, we need to introduce some
necessary notation. Let $X_1, \ldots X_m$ be independent Bernoulli random
variables with $\IP(X_i=1)=p_i$, $W=\sum_{i=1}^m X_i$. Let
\begin{equation}
    \lambda_j=\sum_{i=1}^m p_i^j,\quad j=1,2,\dots,
    \quad \sigma^2=\Var W = \lambda_1-\lambda_2.
\end{equation}
The total variation metric distance between two random variables $X$ and $Y$ is
defined as $$\dtv (\law(X),\law(Y))=\sup_A\ab{\IP(X\in A)-\IP(Y\in A)}$$ where
the supremum is taken over all Borel sets. Note that, if $X$ and $Y$ are
integer-valued, then $\dtv (X,Y)=\ahalf\sum_{i\in\IZ}\ab{\IP(X=i)-\IP(Y=i)}$. We
also define a local metric $$\dloc(\law(X),\law(Y)) =
\sup_{j\in\IZ}|\IP[X=j]-\IP[Y=j]|$$ The notation $\lfloor \cdot \rfloor$ and
$\{\cdot\}$  is used for integral and fractional parts, respectively.

\cite{Ehm1991} gives results for  binomial approximation where the number of
trials equals the number of Bernoulli variables and the success probability is
chosen to match up the first moment. More precisely,
\begin{equation}                                                \label{1}
    \dtv\bigl(\law(W),\Bi(m,p)\bigr) \leq
        \frac{1-p^{m+1}-(1-p)^{m+1}}{(m+1)p(1-p)}
        \sum_{i=1}^m(p_i-p)^2,
\end{equation}
where $p = \lambda_1/m$. Thus, the binomial approximation here is
one-parameter.  Ehm's approach was later extend to Krawtchouk asymptotic
expansion by \cite{Roos2000}.

\citet[p.~190]{Barbour1992} treated the binomial distribution as a two-parameter
approximation. Their result was improved by \citet[Section~4]{Cekanavicius2001},
who showed that
\begin{equation}                                                 \label{2}
\begin{split}
    &\dtv\bigl(\law(W),\Bi(n, p)\bigr) \\
    &\qquad\leq
        \frac{4}{1- p}\min\bigg(1,\frac{\sqrt{\ee}}{\sigma}\bigg)
        \bigg(\frac{\lambda_3}{\lambda_1}
        -\frac{\lambda_2^2}{\lambda_1^2} \bigg)
        + \frac{}{}
            \frac{\lambda_2\{\lambda_1^2/\lambda_2\}}
                {\lambda_1(1-p)n}
        +\IP(W>n)
\end{split}
\end{equation}
Here $n=\lfloor \lambda_1^2/\lambda_2\rfloor$, $ p=\lambda_1/n$.  Note that
\cite{Cekanavicius2001} (as well as \cite{Barbour1992} and \cite{Soon1996}) in
formulations of their results overlooked the term $\IP(W>n)$, which is necessary
because the support of $W$ is typically larger than the support of the
approximating binomial distribution.

It is easy to see that both estimates \eq{1} and \eq{2} are small if all $p_i$
are close to each other. On the other hand, the second estimate can be sharper
than the first one. Indeed, let $p_i=1/2$ for $i\leq m/2$ and $p_i=1/3$
otherwise. Then the right-hand side of \eq{1} equals  some absolute constant
$C_1$, meanwhile the right-hand-side of \eq{2} after application of Chebyshev's
inequality becomes $C_2m^{-1/2}$.

Note that two-parameter binomial approximations are also applied in settings
with dependence, see \cite{Soon1996} and \cite{Cekanavicius2007}.
\cite{Rollin2008a} used a shifted $\Bi(n,1/2)$ to approximate sums of locally
dependent random variables.

In this article we study shifted binomial approximation where the shift, the
number of trials, and the success probability are selected to match up the first
three moments of the shifted binomial and the Poisson-binomial. We then give an
upper bound on the approximation error by adapting Stein's method to the shifted
binomial distribution. This is---to the best of our knowledge---the first time
Stein's method is used to approximate by a distribution that fits the first
three moments. We also discuss the results of a numerical study showing that a
shifted binomial approximation is typically more accurate than the Poisson or
the other standard binomial approximations discussed in \cite{Soon1996} and
\cite{Ehm1991}.

At the end of the article we describe the motivating statistical application in
health-care provider profiling that led to the need for a more accurate
approximation. See \cite{Pekoz2009} for more detail on the application. An introduction to the use of Bayesian hierarchical models for
healthcare provider profiling can be found in \cite{Ash2003}.

Stein's method was introduced in the context of normal approximation by
\cite{Stein1972} and developed for the Poisson distribution by \cite{Chen1974}
and \cite{Chen1975}. The method is particularly interesting since results in the complex
setting of dependent random variables are often not much more difficult to
obtain than results for independent variables. \cite{Barbour1992} details how
the method can be applied to Poisson approximations, \cite{Ehm1991} and
\cite{Loh1992}, respectively, apply the method to binomial and multinomial
approximations, \cite{Barbour1992c} and \cite{Barbour2001c} to compound Poisson
approximation, \cite{Barbour1992b} to Poisson process approximation,
\cite{Pekoz1996} to geometric approximation, and discussion of the many other
distributions and settings the technique can be applied can be found in, for
example, \cite{SteinsMethod2005a} and \cite{Reinert2005}.  An elementary
introduction to Stein's method can be found in Chapter 2 of \cite{Ross2007}.

This paper is organized as follows.  In Section 2 we give the main approximation
theorems by adapting Stein's method to the shifted binomial distribution and in
Section 3 we  prove these results. In Section 4 we discuss numerical results
illustrating the accuracy of several approximations, and in Section 5 we discuss
the statistical application in Bayesian hierarchical modeling that motivated our
initial interest in this approximation.

\section{Main results}

Let $Y$ be a shifted binomial random variable with parameters $n$, $p$ and
integer shift $s$, that is,
\begin{equation}                                            \label{3}
    Y \sim \Bi(n,p) * \delta_s
\end{equation}
where $*$ denotes convolution of measures and $\delta_s$ the measure with mass
$1$ at $s$. In this paper we study the approximation of $W$ using $Y$ with
parameters $n$, $p$ and $s$ chosen so that the first three moments of $W$ and
$Y$ are approximately equal.  Due to the integer nature of $n$ and $s$, it will
not always be possible to exactly match the three moments---so we match them as
closely as possible. We first estimate these parameters, and then give a theorem
bounding the approximation error.   It is easy to check that
\begin{align*}
    \IE Y&=np+s,
    &\Var Y&=np(1-p),
    &\IE (Y-\IE Y) ^{3}& = (1-2p)\Var Y,\\
    \IE W & =\lambda_1,
    &\Var W &=\lambda_1 -\lambda_2,
    &\IE (W-\IE W)^{3} &= \lambda_1 -3\lambda_2 +2\lambda_3.
\end{align*}
In order to find the values $n$, $p$ and $s$, that match the moments best under
the constraint on $n$ and $s$ are integer valued, let us first solve the system
of equations $\IE W =\IE Y $, $\Var W =\Var Y $, $\IE  (W-\IE W )^3 =\IE (Y-\IE
Y )^3 $ for real-valued $n^*$, $p^*$ and $s^*$. The system of
equations
\begin{align*}
    s^*+n^*p^* &=\lambda_1,\\
    n^*p^*(1-p^*)&= \lambda_1 -\lambda_2,\\
    n^*p^*(1-p^*)(1-2p^*)& =\lambda_1 -3\lambda_2 +2\lambda_3,
\end{align*}
yields the solution
\begin{equation}                                                \label{4}
    p^* =\frac{\lambda_2-\lambda_3}{\lambda_1-\lambda_2},
    \qquad
    n^*  = \frac{\lambda_1 -\lambda_2}{p^*(1-p^*)},
    \qquad
    s^* =\lambda_1-n^*p^*.
\end{equation}
We choose now
\begin{equation*}
    n=\lfloor n^*\rfloor, \qquad
    s=\lfloor s^*\rfloor, \qquad
    p=\frac{n^*p^*+\{s^*\}}{n} = p^* + \frac{\{n^*\}p^*+\{s^*\}}{n}
\end{equation*}
(in the last expression we indeed divide by $n$ and not by $n^*$) and then let
$Y$ be as in~\eq{3}. Although $p$ is real valued and therefore
does not need any rounding correction with respect to $p^*$, a small
perturbation is still necessary in order to fit the mean exactly, which is
crucial to obtain better rates of convergence. For convenience, whenever we use
a variable $p$ (or $p_i$, $p^*$ etc.)\ to denote a probability, the variable $q$
(or $q_i$, $q^*$ etc.)\ will denote the counter probability~$1-p$.
Let $v =\sum_{i=1}^m (p_i\wedge q_i)$. Then our main result is the following.

\note{I reworked the bounds; see below. It's the most elegant version I could
come up with. Hope you like it. The issue with the big
constants should be solved now. I also used that $1\wedge \frac{1}{x_+}=
\frac{1}{1\vee x}$ which I think makes the expressions a bit neater.\\ @great -- you have a knack for this type of thing. Erol}

\begin{theorem}\label{thm}
Suppose $X_1,\dots,X_m$ are independent Bernoulli random variables with
$\IP(X_i=1)=p_i$.
With the definitions above, we have
\begin{equation}                                                \label{5}
\dtv \bigl(\law(W),\Bi(n,p)*\delta_s\bigr) \leq K(4A_1 +2A_2)+ \eta,
\end{equation}
where
\begin{align}
 K &= \frac{1-p^{n+1}-q^{n+1}}{\sigma^2},                 \label{6}\\
 A_1 & = \frac{\sigma^2(\lambda_3-\lambda_4)
        -(\lambda_2-\lambda_3)^2}{\sigma^2\big(1\vee (v/2-1)\big)},
 & A_2 & = \frac{\lambda_1[\{n^*\}+\{s^*\}]+n\{s^*\}}{n},\notag\\
 \eta &= (s\max_{i\leq s}p_i)\wedge e^{-\sigma^2/4}
    \,+\, \bigr((m-n-s)\max_{i>n+s}p_i\bigr)\wedge e^{-\sigma^2/4+1}.\notag
    \kern-10em
\end{align}
Furthermore,
\begin{equation}                                                \label{7}
    \dloc \bigl(\law(W),\Bi(n,p)*\delta_s\bigr)
        \leq K(8A_3 + 4A_4) +  \eta,
\end{equation}
where
\begin{equation*}
A_3 = \frac{\sigma^2(\lambda_3-\lambda_4)
        -(\lambda_2-\lambda_3)^2}{\sigma^2\big(1\vee (v/3-2)\big)^{3/2}},
\qquad
A_4  = \frac{\lambda_1[\{n^*\}+\{s^*\}]+n\{s^*\}}{n (1\vee (v-1))^{1/2}}.
\end{equation*}
\end{theorem}
\note{Note that $A_1$ and $A_2$ are so that typically they are of order $1$ and
$A_3$ and $A_4$ are of order $n^{-1/2}$.}

If $p_1, p_2, \ldots , p_m$ are such that for some fixed $p$ we have, for all
$i$, that either $p_i=p$ or $p_i=1$, then $W$ and $Y$ have the same shifted
binomial distribution and $\dtv (\law(W),\law(Y))=0$. In this case after
omitting $\eta$ the right-hand sides of \eq{5} and \eq{7} also both equal zero.
Dropping negative terms, using $(\lambda_3-\lambda_4)\leq \sigma^2\leq v$ and
$1\vee (av-b)\geq av/(1+b)$, and replacing
all fractional parts by unity we obtain the following simplified bounds.

\begin{corollary} Under the conditions of Theorem~\ref{thm}, we have
\note{Old version:
\begin{equation*}
    \dtv \bigl(\law(W),\Bi(n,p)*\delta_s\bigr)
    \leq\frac{16}{\sigma^2} \left(1+ \frac{\lambda_1}{4n}
        +\frac{1}{8}\right)+2\ee^{-\sigma^2/4+1}
\end{equation*}}
\begin{equation*}
    \dtv \bigl(\law(W),\Bi(n,p)*\delta_s\bigr)
    \leq \frac{17+2\lambda_1n^{-1}}{\sigma^2}+2e^{-\sigma^2/4+1},
\end{equation*}
and
\note{Old version:
\begin{equation*}
\dloc \bigl(\law(W),\Bi(n,p)*\delta_s\bigr)
 \leq \frac{162}{\sigma^2\sqrt{v}} \left(1+
\frac{\lambda_1}{4n}+\frac{1}{8}\right)+2\ee^{-\sigma^2/4+1}
\end{equation*}}
\begin{equation*}
\dloc \bigl(\law(W),\Bi(n,p)*\delta_s\bigr)
 \leq \frac{222+12\lambda_1n^{-1}}{\sigma^2v^{1/2}} + 2\ee^{-\sigma^2/4+1}.
\end{equation*}
\end{corollary}
\note{
Adrian: I couldn't figure it out either so I changed the equation from
\begin{equation*}
    \dtv \bigl(\law(W),\Bi(n,p)*\delta_s\bigr)
    \leq\frac{4\lambda_1/n+18}{\sigma^2}
        +3\ee^{-\sigma^2/4}
\end{equation*}
- Erol}
\note{Adrian: Now I changed it back to what Vydas had ( I think the 18 was
supposed to be 1/8) I still don't know where he gets the 3 instead of the 2 on
the exponential at the end. --Erol\\
@Erol: I agree with the 2 instead of 3. There was a mistake in Eq.~\eq{765}. It
said 6 instead of 8. I've corrected it. That's the reason the new constant is
a little bigger than 162. --Adrian
}
It is clear from this corollary that when $c<p_i<d$ for all $i$ and for some
absolute constants $c,d$, the order of upper bound on $\dtv
\bigl(\law(W),\Bi(n,p)*\delta_s\bigr)$ is $O(n^{-1})$ while for $\dloc
\bigl(\law(W),\Bi(n,p)*\delta_s\bigr)$ it is $O(n^{-3/2})$. Thus, we obtain a
significant improvement over $O(n^{-1/2})$ which can be obtained by
two-parametric binomial approximation \eq{2} or by a shifted $\Bi(n,1/2)$
distribution as in \cite{Rollin2008a}.

\section{Proof of the main results}

If Stein's method for normal approximation $\N(0,\sigma^2)$ is applied to a
random variable $X$, we typically need to bound the quantity
\begin{equation}
    \IE[\sigma^2f'(X)-Xf(X)]                                    \label{8b}
\end{equation}
for some specific functions $f$, were $X$ is assumed to be centered and $\Var X
= \sigma^2$. This corresponds to fitting the first two moments. If three moments
have to be matched, we need a different approximating distribution and a
canonical candidate would be a centered $\Gamma(r,\lambda)$ distribution. This
would lead to bounding the quantity
\begin{equation}                                                    \label{8}
    \IE\bigl[\bigl(r\lambda^{-2}+\lambda^{-1}X \bigr)f'(X) -
        Xf(X) \bigr],
\end{equation}
(c.f. \citet[Eq.~(17)]{Luk1994}) where the parameters $r$ and $\lambda$ are
chosen to fit the second and third moments of $W$, that is, $\Var W =
r\lambda^{-2}$ and $\IE W^3/\Var W = 2\lambda^{-1}$ (this obviously is only
possible if $W$ is skewed to the right, which we can always achieve by
considering either $W$ or $-W$). One can see that \eq{8} is in some sense a
more general form of \eq{8b}, having an additional parameter for skewness. On
the integers, we can take a shifted binomial distribution as in this article.
Not surprising, the Stein operator for a binomial distribution, shifted to have
expectation $\lambda_1$ (ignoring rounding problems) can be written in a way
similar to \eq{8}; see \eq{9} below. In the following lemma we give the
basic arguments how we can handle expressions of type \eq{8} in the discrete
case for sums of independent indicators where all the involved
parameters are allowed to be continous. We will deal with the rounding problems
in the main proof of Theorem~\ref{thm}.

We need some notation first. For any function $g$, define the operators $\D^k g
(w) := \D^{k-1}g(w+1)-\D^{k-1}g(w)$ with $\D^0 g := g$ and $\T g(w) :=
(g(w+1)+g(w))/2$. Note that $\T\D = \D\T$. We introduce the operator $\T$ in
order to present the Stein operator of the shifted binomial in a symmetrized
form, so that the connection with \eq{8} should become more apparent. For the
choice $p^* = 1/2$, the linear part in the $\D g$ part will vanish, so that the
operator indeed becomes symmetric, hence corresponds to the symmetric
distribution $\Bi(n^*,1/2)$
shifted by $-n^*/2$.

\begin{lemma}\label{lem1}
Let $W$ be defined as before and let
\begin{equation}                                            \label{9}
\hat{\cal B}^* g(w) := \bigl(n^*p^*q^* + (\ahalf - p^*) (w-\lambda_1)\bigr) \D
g(w)-(w-\lambda_1)\T g(w).
\end{equation}
Then, for $n^*$ and $p^*$ defined as in~\eq{4}, we have for any bounded function
$g:\IZ\to\IR$ that
\begin{equation*}
\begin{split}
    \IE\hat{\cal B}^* g(W)
    &= \sum_{i=1}^m (p^*-p_i)p_i^2q_i\IE\D^3 g(W_i)\\
    &= \frac{1}{2\sigma^2}\sum_{i,j=1}^m p_i p_j
            q_iq_j(p_i-p_j)^2 \IE\D^3 g(W_{ij}),
\end{split}
\end{equation*}
where $W_i := W - X_i$ and $W_{ij} := W - X_i - X_j$.
\end{lemma}

\begin{proof} It is easy to prove that, for any bounded function $h:\IZ\to\IR$,
the following identities hold:
\begin{align}
    \IE[(X_i-p_i)h(W)] &= p_iq_i\IE[\D h(W_i)],             \label{10}\\
    \IE[h(W)-\T h(W_i)] &= -(\ahalf-p_i)\IE \D h(W_i) ,     \label{11}\\
    \IE[h(W)-h(W_i)] &= p_i\IE \D h(W_i).                   \label{12}
\end{align}
In what follows summation is always assumed to range over $i=1,\dots,m$. Using
first \eq{10} and then \eq{11} we obtain that
\begin{equation*}
\begin{split}
    \IE[(W-\lambda_1)\T g(W)]
    & = \sum (X_i-p_i)\T g(W)
      = \sum p_iq_i\IE\D\T g(W_i)\\
    & = \sum p_iq_i\IE\D g(W) + \sum p_iq_i(\ahalf-p_i)\IE\D^2 g(W_i).
\end{split}
\end{equation*}
From \eq{10} we also deduce that
\begin{equation*}
\begin{split}
    \IE[(W-\lambda_1)\D g(W)]
    & = \sum p_iq_i\IE\D^2 g(W_i).\\
\end{split}
\end{equation*}
Combining these two identities and recalling that $n^*p^*q^* =
\lambda_1-\lambda_2$,
\begin{equation*}
\begin{split}
    \IE \hat{\cal B}^* g(W) = \sum p_iq_i(p_i-p^*)\IE \D^2 g(W_i).
\end{split}
\end{equation*}
Applying \eq{12} and noting that $\sum p_iq_i(p_i-p^*)= 0$ proves the first
equality.
For the second equality, we proceed with
\begin{equation*}
\begin{split}
    & \sum_{i=1}^m (p_i-p^*)p_i^2q_i\IE\D^3 g(W_i)\\
    &\quad=\frac{1}{\sigma^2}\sum_{i,j=1}^m p_i^2 p_j
q_iq_j(p_i-p_j)\IE\D^3g(W_i)\\
    &\quad= \frac{1}{2\sigma^2}\sum_{i,j=1}^m p_i p_j
            q_iq_j(p_i-p_j)
        \bigl(p_i\IE\D^3g(W_i)-p_j\IE\D^3g(W_j)\bigr) \\
    &\quad= \frac{1}{2\sigma^2}\sum_{i,j=1}^m p_i p_j
            q_iq_j(p_i-p_j)\bigl(p_i\IE\D^3g(W_{ij})+p_ip_j\IE\D^4g(W_{ij})\\
    &\qquad         \kern12em
        -p_j\IE\D^3g(W_{ij})-p_ip_j \IE\D^4 g(W_{ij})\bigr)\\
    &\quad= \frac{1}{2\sigma^2}\sum_{i,j=1}^m p_i p_j
            q_iq_j(p_i-p_j)^2 \IE\D^3g(W_{ij}).                 \qedhere
\end{split}
\end{equation*}
\end{proof}

The following fact was used already in \cite{Rollin2008a} implicitly. We give a
quick proof here. It is a simple extension of the result in \cite{Ehm1991}, and
is necessary, as $W$ may have a larger support than $Y$.

\begin{lemma}\label{lem2}
Let $A\subset \IZ$ and define the operator
    ${\cal B}f(k) := p(n-k)f(k+1)-qkf(k)$.
Let $f:\IZ\to\IR$ be the solution to
\begin{equation}                                                  \label{13}
    {\cal B}f(k) =
     I_{k\in A}-\Bi(n,p)\{A\} \quad \text{if \,$0\leq k \leq n$,}
\end{equation}
and let $f(k) = 0$ for $k\notin\{0,1,\dots,n\}$.
Then, with $K$ as defined in \eq{6},
\begin{equation}                                                    \label{14}
    \|\D f\| \leq K.
\end{equation}
Furthermore, if $A = \{k\}$ for some $k\in\IZ$, we also have
\begin{equation}                                                    \label{15}
    \|f\| \leq K.
\end{equation}

\end{lemma}
\begin{proof}
Note that, for $1\leq k\leq n$, $f(k)$ coincides with the
definition in \cite{Ehm1991}, who showed that
\begin{equation*}
    \sup_{k\in\{1,\dots,n-1\}}\ab{\Delta f(k)}
    \leq \frac{1-p^{n+1}-q^{n+1}}{(n+1)pq} < K.
\end{equation*}
It remains to bound $\D f(0) = f(1)$ and $\D f(n) = - f(n)$ as obviously
$\D f(k) = 0$ if $k<0$ or $k> n+1$.

Let $\mu := \Bi(n,p)$ be the binomial probability measure. Then,
from \citet[p.~189]{Barbour1992} we have that, for $1\leq k\leq n$ and where
$U_k := \{0,1,\dots,k\}$,
\begin{equation}
\begin{split}
    f(k) &
    =\frac{\mu\{A\cap U_{k-1}\}-\mu\{A\}\mu\{U_{k-1}\}}
{kq\mu\{k\}}
    \\ &
    =\frac{\mu\{A\cap U_{k-1}\}\mu\{U_{k-1}^c\}
            -\mu\{A\cap U_{k-1}^c\}\mu\{U_{k-1}\}}
    {kq\mu\{k\}}.
\end{split}
\end{equation}
From this we have that
\begin{equation}
    |f(k)| \leq \frac{\mu\{U_{k-1}^c\}\mu\{U_{k-1}\}}
                {kq\mu\{k\}},
\end{equation}
in particular for $k=1$
\begin{equation}
    |f(1)| \leq \frac{(1-q^n)q}{npq}
    \leq K.
\end{equation}
For the corresponding
bound at the upper boundary, we have again from \citet[p.~189]{Barbour1992}
that
we can also write
\begin{equation}
\begin{split}
    f(k) &
    =-\frac{\mu\{A\cap U^c_{k-1}\}-\mu\{A\}\mu\{U^c_{k-1}\}}
    {(n-k+1)p\mu\{k-1\}}\\
    & =-\frac{\mu\{A\cap U_{k-1}\}\mu\{U_{k-1}^c\}
            -\mu\{A\cap U_{k-1}^c\}\mu\{U_{k-1}\}}
            {(n-k+1)p\mu\{k-1\}}.
\end{split}
\end{equation}
which, applying it for $k=n$, leads to the same bound on $\D f(n)$, so that
\eq{14} follows. The bound on \eq{15} is immediate from the proof
of \citet[Lemma~1]{Ehm1991}.
\end{proof}

\begin{proof}[Proof of Theorem~\ref{thm}]
We need to bound $|\IP[W-s\in A] - \Bi(n,p)\{A\}|$ for any set $A\subset\IZ$.
Let $f:\IZ\to\IR$ be such that \eq{13} holds. Then we can write
\begin{equation}                                            \label{16}
\begin{split}
    &\IP[W-s\in A] - \Bi(n,p)\{A\} \\
    &\qquad = \IP[W-s\in A\setminus\{0,1,\dots,n\}]
    + \IE{\cal B} f(W-s)
\end{split}
\end{equation} and note that this equation holds because $f=0$
outside of $\{0,1,\dots,n\}$.

Let the operator ${\cal B}^*$ be defined as ${\cal B}$ in Lemma~\ref{lem2}
but replacing $n$ by $n^*$ and $p$ by $p^*$, respectively. Let $g(w) :=
f(w-s)$ and recall that $w-s = w-\lambda_1+n^*p^*+\{s^*\}$. Then,
\begin{equation*}
\begin{split}
{\cal B}f(w-s)
    & = {\cal B}^* f(w-s) + \{s^*\}g(w+1) +
(p^*-p)(w-s)\D g(w)\\
    &=: {\cal B}^* f(w-s) + R_1(w).
\end{split}
\end{equation*}
Note further that
\begin{equation*}
\begin{split}
    &{\cal B}^* f(w-s)\\
    &\quad = \bigl(n^*p^*q^* - p^*(w-s-n^*p^*)\bigr) \D f(w-s)
                - (w-s-n^*p^*)f(w-s)\\
    &\quad = \hat{\cal B}^* g(w) -p^*\{s^*\}\D g(w) - \{s^*\}g(w)\\
    &\quad =: \hat{\cal B}^* g(w) + R_2(w),
\end{split}
\end{equation*}
where $\hat{\cal B}^*$ is as in Lemma~\ref{lem1}. Hence,
\begin{equation}
\begin{split}                                               \label{17}
    {\cal B}f(w-s)
    & =  \hat{\cal B}^* g(w) + R_1(w) + R_2(w).
\end{split}
\end{equation}
Let us first deal with the error terms $R_1$ and $R_2$ (which arise only due
to the necessity that $n$ and $s$ have to be integers). Now,
\begin{equation*}
\begin{split}
 R_1(w) + R_2(w) & = (p^*-p)w\D g(w)
    + \bigl(\{s^*\}(1-p^*)-s(p^*-p)\bigr)\D g(w).\\
\end{split}
\end{equation*}
Noting that $\IE[W \D g(W)]= \sum_i p_i\IE \D g(W_i+1)$ and recalling \eq{14},
we have
\begin{equation}
\begin{split}                                                 \label{18}
 |\IE[R_1(W) + R_2(W)]|
    & \leq 2K(\lambda_1\ab{p^*-p}+\{s^*\}) \\
    & \leq 2K(\lambda_1(\{n^*\}+\{s^*\})/n+\{s^*\})
\end{split}
\end{equation}
where we use \begin{equation*}
 \begin{split}
 |\{s^*\}(1-p^*)-s(p^*-p)| &=|s^*(1-p^*)-s(1-p)|\\
  &\leq
|s-s^*| + |sp-s^*p^*| \\  &\leq
\{s^*\}+s^*|p-p^*| +|s-s^*|p \\  &\leq
2\{s^*\} +s^*|p-p^*|\\
& \leq 2\{s^*\} +\lambda_1|p-p^*|.
 \end{split}\end{equation*}

To estimate $\IE\hat{\cal B}^*(W)$ we use Lemma~\ref{lem1}. Estimation of
$\IE\D^3 g(W_{i,j})$ goes along the lines given in \citet[p. 521 and
541]{Barbour2002}. For a random variable $X$, define first
\begin{equation*}
    D^k(X) = \norm{\law(X)*(\delta_0-\delta_1)^{*k}},
\end{equation*}
\note{@Erol: I introduce this notation to make the expressions neater. --Adrian}
where $\|\cdot\|$ denotes the total variation norm when applied to measures.
Note that $D^1(X) = 2\dtv(\law(X),\law(X+1))$. We can decompose $W_{i,j} =
S_{i,j,1}+S_{i,j,2}$ in such a
way, that both sums of the $(p_i\wedge q_i)$ corresponding to $S_{i,j,1}$ and
$S_{i,j,2}$
are greater or equal to $v/2 - v^*$, where $v^* = \max_{1\leq i \leq
m} (p_i\wedge q_i)$. We have
\begin{equation}                                            \label{19}
\begin{split}
 |\IE \D^3 g(W_{i,j})|
    &\leq \|\D g \|\, D^2(W_{i,j})
    \leq \|\D g\| D^1(S_{i,j,1})D^1(S_{i,j,2})\\
    &\leq \frac{4K}{1\vee (v/2-1)}. 
\end{split}
\end{equation}
In the last line we used \citet[Proposition 4.6]{Barbour1999}
and \citet[p.~521, Estimate (4.9)]{Barbour2002}. 

So, starting from \eq{16}, then using identity \eq{17} along with
Lemma~\ref{lem1} and estimate \eq{19} and also estimate \eq{18}, we obtain
\begin{equation*}
\begin{split}
&|\IP[W-s\in A]-\Bi(n,p)\{A\}|\\
&\quad\leq
\frac{4K}{2\sigma^2(1\vee (v/2-1))}\sum_{i,j}
p_ip_jq_iq_j(p_i-p_j)^2
 \\
&\qquad+ 2K(\lambda_1(\{n^*\}+\{s^*\})/n+\{s^*\})+\IP[W<s] +\IP[W>n+s].
\end{split}
\end{equation*}
Note now that
\begin{equation*}
 \frac{1}{2}\sum_{i,j}p_ip_jq_iq_j(p_i-p_j)^2 =
(\lambda_1-\lambda_2)(\lambda_3-\lambda_4) - (\lambda_2-\lambda_3)^2.
\end{equation*}

Consequently, to complete the proof for the total variation distance one needs
to estimate tails of $W$. Note that $X_i-p_i$ satisfies Bernstein's inequality
with parameter $\tau=1$. Therefore,
\begin{equation*}
    \IP(W<s)=\IP(W-\lambda_1<s-\lambda_1)
    \leq\Exponent{-\frac{\sigma^4} { 4\sum p_j(1-p_j)^2}}
    \leq \exponent{-\sigma^2/4}.
\end{equation*}
Similarly, by applying estimate
\begin{equation*}
    \IP(W-\lambda_1>x)\leq\Exponent{\frac{\sigma^2}{4}-\frac{x}{2}},
\end{equation*}
see equation (4.3) from \cite{Arak1988}, we get
\begin{equation}
    \IP(W>n+s)\leq \exponent{-\sigma^2/4+1}.
\end{equation}
Estimate $\IP(W<s)\leq s\max{i<s}p_i$ is straightforward.

To obtain result for the $\dloc$ metric, the proof is similar, except that we
now have $A={k}$ for some $k\in\IZ$ and bound \eq{15}. We need some refinements
of the estimates of $\IE[R_1(W)+R_2(W)]$ and $\IE\hat{\cal B}^*(W)$. Similar
to~\eq{19},
\begin{equation*}
|\IE\D g(W_i)| \leq \norm{g}D^1(W_i)
\leq  \frac{2K}{(1\vee (v-1))^{1/2}}
\end{equation*}
and, choosing $S_{i,j,k}$, $k=1,2,3$, so that the corresponding $(p_i\wedge
q_i)$ sum up to
at least $(v/3-2v^*)$,
\begin{equation}                                                \label{765}
\begin{split}
|\IE\D^3 g(W_{i,j})| & \leq \| g \|\,
D^3(W_{i,j})
\leq \norm{g} \prod_{k=1}^3 D^1(S_{i,j,k})
\leq \frac{8K}{(1\vee (v/3-2))^{3/2}}.
\end{split}
\end{equation}
Plugging these estimates into the corresponding inequalities, the final estimate
\eq{7} is easily obtained.
\end{proof}

\section{Numerical Results}

In this section we study the sum of Bernoulli random variables
$X_1,\dots,X_{100}$ with uniformly spread probabilities from 0 to some parameter
$M$, so that $p_i= iM/(101)$, $i=1,2,\dots,100$. We analytically compute the
exact distribution of $W=\sum_{i=1}^{100}X_i$ and then 
the exact total variation distance between $W$ and several different
approximations for different values of $M$. Figure~\ref{fig} shows a graph of
the exact total variation approximation error for several different
approximations versus $M$, referred to in the graph on the $X$-axis as the
``maximum probability.'' In the graph ``Poisson'' is the standard Poisson
approximation where the parameter is chosen to match the first moment.
``Binomial'' refers to a binomial approximation where a number of trials $n $ is
fixed to equal 100 but the probability of success $p$ is chosen to match the
first moment (this is the approximation studied in \cite{Ehm1991}). ``Shifted
Poisson'' refers to the approximation where a constant is added to a Poisson
random variable and the two parameters -- the constant and the Poisson rate --
are chosen to match the first two moments (this is the approximation studied in
\cite{Cekanavicius2001}).  The ``Normal'' approximation is the standard normal
approximation to the binomial distribution using the continuity correction. ``2
parameter binomial'' refers to the approximation where the two binomial
parameters $n$ and $p$ are chosen to match the first two moments (this is the
approximation studied in \cite{Soon1996}).  Finally, ``shifted binomial'' refers
to the approximation we propose in this paper -- where the shift, the number of
trials, and the probability of success are chosen to match the first three
moments.

We see in Figure~\ref{fig} that the normal approximation performs well when
probabilities are widely spread out but performs very poorly when probabilities
are very small. We see that the Poisson, shifted Poisson, and binomial
approximations are best for small probabilities but not otherwise. The two
parameter binomial approximation is quite good, but the shifted binomial
approximation performs the best over the widest range of values of $M$.  Since
the value of $M$ can be viewed as varying widely in our statistical application,
this would be the preferred approximation.

\begin{figure}
\begin{center}
\includegraphics [width=0.9\hsize]{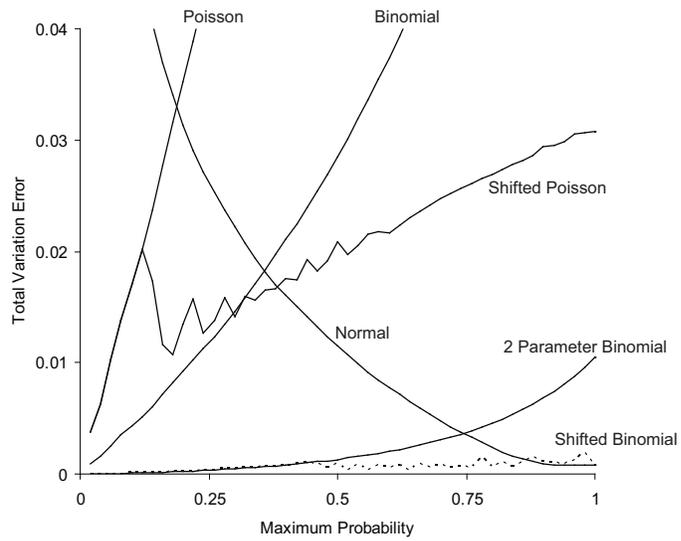}
\end{center}
\caption{\label{fig} Exact total variation distance error between $W$ using
$p_i=iM/101$, $i=1,\dots,100$, and six approximations as a
function of  the maximum probability $M$.}
\end{figure}

In summary, we see that over a range of different Poisson-Binomial random
variables that the shifted binomial approximation performs very well -- usually
better than the other two standard binomial approximations studied previously in
the literature. The advantage of the shifted binomial approximation seems to
increase as the spread among the Bernoulli probabilities increases.

\section{Application to Bayesian Hierarchical Modeling}

The study of shifted binomial approximations is motivated by a statistical
problem (see \cite{Pekoz2009}) of ranking a large number of hospitals with respect to quality as
measured by the risk of adverse events at the hospitals. Let $X_{ij}$ be a
binary data variable that equals 1 if adverse event of a particular type happens
to patient $i$ in hospital facility $j$, and equals zero otherwise. We are
interested in the following model where $X_{ij}$ are the data values, $p_{ij}$
are known constants, and $\theta_j$,  and $\sigma ^ 2$ are unknown
parameters that we would like to estimate:
\note{@Erol: The choice $q_ij$ is now a bit unfortunate, as we use $q=1-p$ in
the first part of the paper. Can we change this? --Adrian\\
ok!-Erol}

\begin{equation*}
    X_{ij}\,|\,p_{ij},\theta_j\sim \mathop{\mathrm{Be}}(\mathrm{logit^{-1}}(\mathrm{logit}(p_{ij})+\theta_j))
\end{equation*}
where
\begin{gather*}
\theta_j\,|\, \sigma ^ 2 \sim \N(0, \sigma ^ 2)
\end{gather*}
In this model $p_{ij}$ is a risk-adjusted probability that has been previously
calculated by taking into account various patient specific indicators and it
represents the chance patient $i$ would have an adverse event  at a typical
hospital. The parameter $\theta_j$ is a hospital specific factor that increases
or decreases the probability of an adverse event for its patients. Hospitals
with a high value of $\theta_j $ are poorly performing hospitals. Our goal is to
rank hospitals by the values of $\theta_j.$ The standard Bayesian hierarchical
modeling approach is to put prior distributions on the unspecified parameters
and estimate the posterior means of all the parameters conditional on the data.

The difficulty in this situation is that the values of $X_{ij}$ and $p_{ij}$ are
both confidential and are too numerous to conveniently transmit from each of the
hospitals to the main research facility that would be performing the analysis.
We need a method for summarizing each of these so that each facility only needs
to report a few summary statistics.  In our application we have thousands
of hospitals, thousands of people in each hospital and a number of different
types of adverse events. A rough approximation of $5,000$ hospitals with $1,000$
people each  yields a total of $5,000\times 1,000 = 5,000,000$ random
variables---too many to be conveniently computable by standard software.

To circumvent this difficulty we propose that
each hospital aggregate its patients and compute $Y_j = \sum_i
X_{ij}$, the number of people in hospital $j$
who have an adverse event.  We then use the shifted binomial approximation above
for $Y_{j}$. This will then yield a total of $5,000$ random variables---much
more easily manageable computationally.

To implement the
approximation, in the preparation stage, hospital $j$ also stores and
submits the values of $\lambda_{jm}\equiv \sum_{i}p^m_{ij}$ for $m=1,2,3$ and all~$j$.
Then we can easily compute the shifted binomial approximation to $Y_{j}$ from
these as a function of~$\theta_j$. This results in the following model:
\begin{gather*}
    \theta_j\,|\, \sigma ^ 2 \sim \N(0, \sigma ^ 2),  \\
    Y_{j}-s_{j}\,|\,\theta_j, n_{j}, p_{j},
         \sim \Bi(n_{j},\mathrm{logit^{-1}}(\mathrm{logit}(p_{j})+\theta_j))
\end{gather*}
with
\begin{equation*}
    p_{j} =\frac{\lambda_{j2}-\lambda_{j3}}{\lambda_{j1}-\lambda_{j2}},
    \qquad
    n_{j} = \frac{\lambda_{j1} -\lambda_{j2}}{p_{j}(1-p_{j})},
    \qquad
    s_{j} =\lambda_{j1}-n_{j}p_{j}
\end{equation*}
being the parameters for the shifted binomial approximation designed to match up three moments.
\begin{remark}
Though the binomial distribution is not defined for fractional values of the parameter $n$, we can use a fractional parameter in the likelihood function for the data to obtain in some sense an interpolation of the likelihood functions under the two closest binomial models having integer parameters. For many statistical parameter estimation software packages using likelihood-based approaches, such as maximum likelihood or the Metropolis algorithm, such fractional values of the binomial parameter $n$ can be used this way to yield better approximations.

For example in the simple model for the data $ X|\,n, p \sim \Bi(n,p)$, the likelihood function for the data as a function of the unknown parameter $p$ is
$L(p)\propto p^X (1-p)^{n-X}.$  Under likelihood-based approaches this function is all that is used from the model to estimate the parameters, and so the use of non-integer $n$ the function $L(p)$ can be viewed as yielding an interpolation of the likelihood functions $L_1(p)\propto p^X (1-p)^{\lceil n\rceil -X}$ and $L_2(p)\propto p^X (1-p)^{\lfloor n\rfloor -X}$.
\note{@Erol: Not quite clear to me what you mean by that. How is the $\Bi(n,p)$
defined for non-integer $n$?
\\
@ It's not defined -- but we can define the likelihood function of the data}
\end{remark}

\section{Acknowledgments}
VC, EP and AR would like to express gratitude for the gracious hospitality of
Andrew Barbour and Louis Chen during a visit to the National University of
Singapore in January 2009 (where a portion of this work was completed), as well
as gratitude for generous support from the Institute for Mathematical Sciences
of the National University of Singapore.  EP and MS would like to thank the
Center for Organization, Leadership and Management Research at the Veterans'
Health Administration also for generous support. Thanks are also due to the
referee for many valuable comments that have led to significant improvements in the paper.

\setlength{\bibsep}{\smallskipamount}
\def\bibfont{\small}


\end{document}